\documentclass[11pt,reqno]{amsart}
\usepackage{eucal,fullpage,times,amsmath,amsthm,amssymb,mathrsfs,stmaryrd,color,enumerate,accents}
\usepackage[all]{xy}
\usepackage{url}
\usepackage[colorlinks=true,hyperindex, pagebackref=true, citecolor=cyan]{hyperref}

\newcommand{\cosimp}[3]{\xymatrix@1{#1 \ar@<.4ex>[r] \ar@<-.4ex>[r] & {\ }#2 \ar@<0.8ex>[r] \ar[r] \ar@<-.8ex>[r] & {\ } #3 \ar@<1.2ex>[r] \ar@<.4ex>[r] \ar@<-.4ex>[r] \ar@<-1.2ex>[r] & \cdots }}

\newcommand{\colim}{\mathop{\mathrm{colim}}}

\newcommand{\adjunction}[4]{\xymatrix@1{#1{\ } \ar@<0.3ex>[r]^{ {\scriptstyle #2}} & {\ } #3 \ar@<0.3ex>[l]^{ {\scriptstyle #4}}}}

\begin{document}

\newtheorem{theorem}{Theorem}[section]
\newtheorem*{theorem*}{Theorem}
\newtheorem*{definition*}{Definition}
\newtheorem{proposition}[theorem]{Proposition}
\newtheorem{lemma}[theorem]{Lemma}
\newtheorem{corollary}[theorem]{Corollary}

\theoremstyle{definition}
\newtheorem{definition}[theorem]{Definition}
\newtheorem{question}[theorem]{Question}
\newtheorem{remark}[theorem]{Remark}
\newtheorem{warning}[theorem]{Warning}
\newtheorem{example}[theorem]{Example}
\newtheorem{notation}[theorem]{Notation}
\newtheorem{convention}[theorem]{Convention}
\newtheorem{construction}[theorem]{Construction}
\newtheorem{claim}[theorem]{Claim}

\title{Finiteness of \'etale fundamental groups by reduction modulo $p$}
\author{Bhargav Bhatt}
\author{Ofer Gabber}
\author{Martin Olsson}

%\begin{abstract}
%We introduce a technique to deduce finiteness of \'etale fundamental groups of varieties in characteristic $0$ from finiteness for the corresponding groups for the characteristic $p$ specialization of the variety. The main novelty is that we do not require the characteristic $p$ specializations to have uniformly bounded fundamental groups: a bound that grows polynomial in $p$ suffices. We apply this to deduce a finiteness result of Xu, proven using MMP techniques, from an analogous recent result of Carvajal--Rojas-Schwede-Tucker in $F$-singularity theory.
%\end{abstract}

\maketitle

\begin{abstract}
We introduce a spreading out technique to deduce finiteness results for \'etale fundamental groups of complex varieties by characteristic $p$ methods, and apply this to recover a finiteness result proven recently for local fundamental groups in characteristic $0$ using birational geometry. 
\end{abstract}

\section{Introduction}
\label{sec:Intro}

The main goal of this paper is to introduce a spreading out technique for deducing finiteness results for \'etale fundamental groups of varieties in characteristic $0$ from the finiteness of the \'etale fundamental groups of the characteristic $p$ specializations of the varieties. Our motivation was to deduce a recent finiteness result of Xu \cite[Theorem 1]{ChenyangPi1}, proven using fundamental advances \cite{BCHM} in the minimal model program, from the characteristic $p$ analog, which was established very recently by Carvajal--Rojas, Schwede, and Tucker  \cite[Theorem 5.1]{CRST:pi1} and is considerably more elementary. Xu's result, which answers an algebraic version of a question of Koll\'ar \cite[Question 26]{Kollar},  asserts:
\begin{theorem}[{\cite[Theorem 1]{ChenyangPi1}}]
\label{thm:XuIntro}
Let $X/\mathbf{C}$ be a finite type $\mathbf{C}$-scheme and let $x\in X(\mathbf{C})$ be a closed point.  Assume that $X$ has klt singularities in the sense of \cite[Definition 2.8]{KollarSingMMP}, and that $\dim(X) \geq 2$. Then the \'etale local fundamental group 
\[ \pi_1^{loc}(X,x) := \pi_1(\mathrm{Spec}(\widehat{\mathcal{O}_{X,x}}) - \{x\})\]
 is finite. More generally, the same is true for any finite type normal pointed $\mathbf{C}$-scheme  $(X,x)$ of dimension $\geq 2$ that admits an effective $\mathbf{Q}$-Weil divisor $\Delta$ with the pair $(X,\Delta)$ being klt in a neighborhood of $x$.
\end{theorem}

We relate this to fundamental groups in positive characteristic as follows.  Spread the pair $(X,x)$ out to a pointed scheme $(X_A,x_A)$ over a finitely generated $\mathbf{Z}$-subalgebra $A \subset \mathbf{C}$, and for a geometric point $\bar y\rightarrow \mathrm{Spec}(A)$ let $(X_{\bar y}, x_{\bar y})$ denote the base change of $(X_A, x_A)$.  
  Our main result on spreading out is the following:
\begin{theorem}
\label{thm:MainThmIntro}
 Assume that there exists a positive integer $c$ and a dense open subset $U\subset \mathrm{Spec}(A)$ such that 
 for all geometric points $\bar y\rightarrow U$ with image a closed point of $U$ we have 
\begin{equation}
\label{eq:PolyBoundIntro}
\# \pi^{loc}_{1}(X_{\bar y},x_{\bar y})^{(p_{\bar{y}})} \leq p_{\bar{y}}^c, 
\end{equation}
where $p_{\bar{y}}$ is the characteristic of the residue field of $\bar y$, and $\pi^{loc}_{1}(X_{\bar y},x_{\bar y})^{(p_{\bar{y}})}$ is the maximal prime-to-$p_{\bar{y}}$ quotient of $\pi_1^{loc}(X_{\bar{y}}, x_{\bar{y}})$. Then $\pi_1^{loc}(X,x)$ is finite.
\end{theorem}

Note that if the right side of the bound in \eqref{eq:PolyBoundIntro} above was independent of the residue characteristic, then Theorem~\ref{thm:MainThmIntro} would be unsurprising.  The main novelty here is that a bound in characteristic $p$ that grows polynomially in $p$ suffices to deduce finiteness in characteristic $0$. In fact, much weaker growth bounds on $\#\pi^{loc}_1(X_{\bar{y}},x_{\bar{y}})$ suffice for the result above, but we only need the polynomial one in our application. The proof of Theorem~\ref{thm:MainThmIntro} has two ingredients that are themselves perhaps of independent interest: a constructibility result for the tame local fundamental groups in arithmetic families without any restrictions on singularities (see Theorem~\ref{thm:constructibility}), and an abstract finiteness criterion for a profinite group (see Proposition~\ref{prop:ProfiniteGroupFinitenessCriterion}).

%\begin{remark}
%Theorem~\ref{thm:MainThmIntro} admits an evident analogue for varieties and pairs. EXPAND
%\end{remark}

To apply Theorem~\ref{thm:MainThmIntro} to reprove Theorem~\ref{thm:XuIntro}, recall that a fundamental result in the theory of $F$-singularities tells us that klt singularities in characteristic $0$ specialize to strongly $F$-regular singularities in characteristic $p$; see \cite{HaraRational,HaraWatanabe,Takagi} and \cite[Theorem 2]{SmithGuide} for the relevant definitions and further references. Thus, one needs the finiteness of local fundamental groups of strongly $F$-regular singularities in characteristic $p$ with a bound of the form \eqref{eq:PolyBoundIntro} above. The finiteness is provided by the following recent result \cite{CRST:pi1}, giving the characteristic $p$ analog of Theorem~\ref{thm:XuIntro}.

\begin{theorem}[{\cite[Theorem 5.1]{CRST:pi1}}]
\label{thm:CSTIntro}
Let $k$ be an algebraically closed field of positive characteristic $p$, and let $X/k$ be a finite type $k$-scheme. Let $x \in X(k)$ be a point such that $X$ is strongly $F$-regular at $x$ with dimension $\geq 2$.
 Then $\pi_1^{loc}(X,x)$ is finite.
\end{theorem}

To deduce \ref{thm:XuIntro} from \ref{thm:MainThmIntro} and \ref{thm:CSTIntro} it remains to find an effective bound on $\pi_1^{loc}(X,x)$, at least when specializing from characteristic $0$ to large characteristics. Such a bound can be obtained from the methods of \cite{CRST:pi1}. Let $A$ be a finitely generated normal $\mathbf{Z}$-algebra and let $(X_A, x_A)$ be a normal pointed flat $A$-scheme.  Fix also an imbedding $A\hookrightarrow \mathbf{C}$ and let $(X_{\mathbf{C}}, x_{\mathbf{C}})$ be the base change to $\mathbf{C}$.  Assume that there exists a $\mathbf{Q}$-divisor $\Delta _{\mathbf{C}}$ on $X_{\mathbf{C}}$ such that the pair $(X_{\mathbf{C}}, \Delta _{\mathbf{C}})$ has klt singularities in a neighborhood of $x_{\mathbf{C}}$. Tucker has shown (unpublished) the following:

\begin{theorem}
\label{thm:generalbound}
There exists a dense open subset $U\subset \mathrm{Spec}(A)$ such that for every geometric point $\bar y\rightarrow U$ mapping to a closed point of $U$ we have
$$
\#\pi_1^{loc}(X_{\bar y},x_{\bar y}) \leq p_{\bar y}^d,
$$
where $d$ is the dimension of $X_{\mathbf{C}}$ and $p_{\bar y}$ is the characteristic of $k(\bar y)$.
\end{theorem}

In fact, in the setup of Theorem~\ref{thm:CSTIntro}, the proof in \cite{CRST:pi1} implies that $\#\pi_1^{loc}(X,x)$ is bounded above by the inverse of the $F$-signature $s(X,x)$, which is a numerical invariant of the singularity at $x$. Theorem~\ref{thm:generalbound} is thus a consequence of finding a suitable lower bound on $s(X_{\bar{y}},x_{\bar{y}})$ as $\bar{y}$ varies through geometric points of $\mathrm{Spec}(A)$;  we supply a proof of the relevant bound, based on ideas from \cite{PolstraTuckerFsig}, in \S \ref{sec:proofs}.

The hypercover techniques that go into proving Theorem~\ref{thm:MainThmIntro} are quite general, and have other applications. For example, we use them to prove a constructibility result for tame local fundamental groups of ``Milnor tubes" around points of a given subvariety of an algebraic variety (see Theorem~\ref{thm:ConstructibilityMilnorpi1}). Using this, in \S \ref{sec:GKP}, we give a direct Galois-theoretic proof of the following result of Greb, Kebekus, and Peternell, yielding a global version of Theorem~\ref{thm:XuIntro}; as a bonus, we can avoid the quasi-projectivity hypothesis present in \cite{GKP}. 

\begin{theorem}[{\cite[Theorem 1.5]{GKP}}]
\label{thm:GKPIntro}
Let $X$ be a connected normal scheme of finite type over $\mathbf{C}$. Assume that the pair $(X,\Delta)$ is klt for some effective $\mathbf{Q}$-Weil divisor $\Delta$ on $X$. Then there exists a finite cover $f:Y \to X$ with $Y$ connected and normal such that $f$ is \'etale over the smooth locus $X_{sm} \subset X$, and $f$ induces an isomorphism $\pi_1(f^{-1}(X_{sm})) \xrightarrow{\simeq} \pi_1(Y)$.
\end{theorem}

Finally, though we do not discuss it here, similar hypercover techniques can also be used, for example, to obtain constructibility results for fundamental groups of open varieties.

\subsection*{Strategy of proof of \ref{thm:MainThmIntro} and an outline of the article}
The proof of Theorem~\ref{thm:MainThmIntro}, which is presented in \S \ref{sec:proofs}, has two  parts. First, in \S \ref{S:section4}, using $h$-descent for \'etale fundamental groups, we establish a constructibility result that identifies the maximal prime-to-$p_{\bar{y}}$ quotients of $\pi_1^{loc}(X,x)$ and $\pi_1^{loc}(X_{\bar{y}},x_{\bar{y}})$ for $\bar{y}$ varying through geometric points of an open subset of $\mathrm{Spec}(A)$. The key idea is to describe either group completely in terms of components of the special fibre of a (truncated) hypercover by normal crossings pairs, and then exploit specialization properties of \'etale fundamental groups of smooth projective varieties. The necessary tools are developed separately in \S \ref{sec:tools} in the language of root stacks; we also use the same techniques in \S \ref{S:section3} and \S \ref{sec:GKP} to prove Theorem~\ref{thm:GKPIntro}. Thanks to Theorem~\ref{thm:generalbound}, we then find ourselves in the following situation: the profinite group $G := \pi_1^{loc}(X,x)$ has the property that its maximal prime-to-$p$ quotient $G^{(p)}$ is finite of size $\leq p^c$ for all but finitely many primes $p$. We then prove, in \S \ref{sec:ProfiniteGroup}, the general statement that any such $G$ must be finite.  The main idea is to exploit the fact that if $G$ admits a finite quotient $Q$ whose order $N$ is large, then there must be many primes $p \ll N$ that are coprime to $N$, and thus $Q$ must be a quotient of $G^{(p)}$ for a relatively small such prime $p$; the bound on $G^{(p)}$ then forces $N$ to be small, which is a contradiction.

%\subsection*{Remark} The hypercover techniques used here are quite general.  Though we do not discuss it further in this paper, similar arguments can be used, for example,  to obtain constructibility results for fundamental groups of open varieties.

\subsection*{Acknowledgements}
The key Proposition~\ref{prop:ProfiniteGroupFinitenessCriterion} was discovered during a coversation with Johan de Jong and Peter Scholze at a workshop in Oberwolfach, and we thank them for the discussion. We are also grateful to Karl Schwede and Kevin Tucker for discussions surrounding Theorem~\ref{thm:generalbound}. Finally, we thank Max Lieblich for a useful conversation. During the preparation of this work, BB was partially supported by NSF Grant DMS-1501461 and a Packard fellowship, and MO was partially supported by NSF grants DMS-1303173 and DMS-1601940.

\subsection*{Notation}

We adopt the following largely standard conventions in this article.

\subsubsection*{Groups} For any profinite group $G$ and a prime number $p$, we write $G^{(p)}$ for the projective limit of all continuous finite prime-to-$p$ quotients of $G$; we follow the convention that $G^{(p)} = G$ if $p = 0$. 

\subsubsection*{Categories of finite etale covers} Given a scheme (or stack) $U$, we write $\mathrm{Fet}(U)$ for its category of finite \'etale covers. Given a prime $p$, let $\mathrm{Fet}^{(p)}(U) \subset \mathrm{Fet}(U)$ be the full subcategory spanned by those finite \'etale covers whose Galois closures have degree coprime to $p$ on every connected component of $U$. Note that both $\mathrm{Fet}(U)$ and $\mathrm{Fet}^{(p)}(U)$ are contravariantly functorial in $U$.

\subsubsection*{Local fundamental groups} Given a scheme $X$ and a geometric point $x:\mathrm{Spec}(\kappa) \to X$, write $X_x^{sh}$ for the strict henselization of $X$ at $x$, and let $X_x^{sh, \circ} := X_x^{sh} - \{x\} \subset X_x^{sh}$ be the punctured spectrum. When $X_x^{sh, \circ}$ is connected (for example, if $X$ is geometrically unibranch), write $\pi_1^{loc}(X,x) := \pi_1(X_x^{sh, \circ})$ for the \'etale fundamental group of the punctured spectrum at some unspecified base point; thus, $\pi_1^{loc}(X,x)^{(p)}$ is the profinite group classifying $\mathrm{Fet}^{(p)}(X_x^{sh, \circ})$. By \cite[Corollaire on p. 579]{Elkik}, if $X$ is noetherian, we may replace the strict henselization by its completion in this discussion. 

\subsubsection*{Specializations} Given a scheme $X$ and geometric points $\eta \to X$ and $s \to X$, a specialization $\eta \rightsquigarrow s$ is (by definition) a factorisation of $\eta$ through $X^{sh}_s$. Such a factorisation gives a map $X_\eta^{sh} \to X_s^{sh}$, and allows us to view $\eta$ as a geometric point of $X^{sh}_s$. 

\subsubsection*{Simplicial techniques} The language of hypercovers of schemes, especially in the context of the $h$-topology, will be freely used. In fact, we exclusively work with hypercovers in the $2$-truncated sense (see for example \cite[3.1]{ConradCohomologicalDescent}). This has the advantage of being a more finitistic notion than that of a general hypercover; it suffices for our purposes as the functors we apply to hypercovers (such as $\mathrm{Fet}(-)$) take values in $1$-categories, so their limits are insensitive to replacing a hypercover by its $2$-truncation (as discussed next).

\subsubsection*{Category theory} We often talk about limits and colimits of categories; these are {\em always} intended in the $2$-categorical sense, and basic facts about them, such as the commutation of filtered colimits with finite limits, will often be used without further comment. We also often implicitly use the following observation to convert possibly infinite limits to finite ones: if $\mathcal{C}^\bullet$ is a cosimplicial diagram of categories, and $\mathcal{C}^\bullet_{\leq 2}$ is the corresponding $2$-truncated diagram, then $\lim C^\bullet \simeq \lim C^\bullet_{\leq 2}$.

\section{Some tools}
\label{sec:tools}

In this section, we collect together some tools that will be used in the sequel to prove constructibility results for \'etale fundamental groups in certain families.

%\begin{notation}
%We will use the following device to work with specializations. Given an absolutely integrally closed valuation ring $V$ and a map $\mathrm{Spec}(V) \to X$, we obtain a specialization $\eta \rightsquigarrow s$ of geometric points of $X$ determined by the images of the generic and closed points of $V$ respectively. Conversely, any specialization $\eta \rightsquigarrow s$ of geometric points of $X$ arises by this procedure applied to $V$'s arising as absolute integral closures of henselian dvrs over $k$: the scheme-theoretic image of the map $\eta \to X^{sh}_s$ determining a specialization $\eta \rightsquigarrow s$ is an irreducible closed subset $F$ of $X^{sh}_s$, and the integral closure $V$ in $\eta$ of any discrete valuation on the function field of $F$ centered at $s$ (such as one obtained from the local ring at a generic point of the exceptional divisor on the blowup of $F$ at $s$) provides the desired valuation ring.
%\end{notation}

\subsection{$h$-descent}

Recall that for a scheme $X$, the $h$-topology on the category of finite type $X$-schemes is the Grothendieck topology generated by declaring proper surjective maps and \'etale covers to be coverings.

\begin{lemma}
\label{lem:FetCD}
Let $\pi:X_\bullet \to X$ be a $2$-truncated hypercover of a noetherian scheme $X$ in the $h$-topology. Then pullback induces equivalences
\[ \mathrm{Fet}(X) \simeq \lim \mathrm{Fet}(X_\bullet) \quad \text{and} \quad \mathrm{Fet}^{(p)}(X) \simeq \lim \mathrm{Fet}^{(p)}(X_\bullet).\]
The same is also true if $X$ is a noetherian algebraic stack.
\end{lemma}

In other words, specifying a finite \'etale cover of $X$ is equivalent to specifying its pullback to $X_0$ together with an isomorphism of the two pullbacks to $X_1$ that satisfy the cocycle condition over $X_2$.

\begin{proof}
This is immediate from the fact that finite \'etale morphisms satisfy effective descent for the $h$-topology \cite[\S IX, Theorem 4.12]{SGA1}.
\end{proof}

The next result is a standard ingredient in \'etale cohomology, featuring notably in the proof of the proper base change theorem.

\begin{lemma}
\label{lem:FetProperHenselian}
Let $R$ be a noetherian ring which is henselian with respect to an ideal $I$. Let $f:X \to \mathrm{Spec}(R)$ be a proper morphism with $X$ an algebraic stack. Setting $X_0 = f^{-1}(\mathrm{Spec}(R/I))$, restriction induces an equivalence
\[ \mathrm{Fet}(X) \simeq \mathrm{Fet}(X_0).\]
\end{lemma}
\begin{proof}
When $X$ is a scheme, this is \cite[Tag 0A48]{StacksProject}. The case of proper stacks then follows from Lemma~\ref{lem:FetCD} and Chow's lemma for stacks, see \cite{OlssonChow} and \cite[Tag 0CQ8]{StacksProject}. 
\end{proof}

\begin{remark}
Lemma~\ref{lem:FetProperHenselian} implies that for a proper morphism $f:X \to S$ of noetherian schemes and a specialization $\eta \rightsquigarrow s$ of geometric points of $S$, we obtain an induced functor
\[ \mathrm{Fet}(X_s) \xleftarrow{\simeq} \mathrm{Fet}(X \times_S S_s^{sh}) \to \mathrm{Fet}(X_\eta).\]
We often refer to this functor (or the induced map on fundamental groups in the presence of a base point) informally as the specialization map.
\end{remark}

\subsection{Root stacks}
\label{ss:RootStacks}

Root stacks, originally introduced in \cite{Cadman}, provide an alternative and quicker route to some results traditionally formulated using logarithmic geometry. In this section, we recall some basic results in this theory. Everything that follows (except perhaps Proposition~\ref{prop:Logpi1Constructibility}) is presumably well-known to the experts. As a general reference, we use \cite[Appendix A]{LieblichOlsson}. We first introduce the objects that will be used throughout this section.

\begin{notation}
\label{not:RootStack}
Fix a prime $p$ and a  scheme $X$ over $\mathbf{Z}_{(p)}$. Fix a finite set $I$ and pairs $\{(L_i,s_i)\}_{i \in I}$ where $L_i$ is a line bundle on $X$  and $s_i :L_i\rightarrow \mathcal{O}_X$ is a morphism of $\mathcal{O}_X$-modules.  Equivalently we can view $s_i$ as an element of $H^0(X, L_i^\vee )$.   As in \cite[\S A.4]{LieblichOlsson}, each pair $(L_i,s_i)$ corresponds to a map $X \to [\mathbf{A}^1/\mathbf{G}_m]$, and scaling $s_i$ by a unit defines an isomorphic map. Taking products gives a map $X \to [\mathbf{A}^1/\mathbf{G}_m]^I$ whose $i$-th component corresponds to $(L_i,s_i)$. For each positive integer $\ell$ coprime to $p$, pulling back the map $[\mathbf{A}^1/\mathbf{G}_m]^I \to [\mathbf{A}^1/\mathbf{G}_m]^I$ given by $(z_i) \mapsto (z_i^\ell)$ gives a cover $\mathcal{X}_\ell \to X$ with $\mathcal{X}_\ell$ a Deligne-Mumford stack called the \emph{$\ell $-th root stack}; note that the former map is an isomorphism over the open substack $[\mathbf{G}_m/\mathbf{G}_m]^I \hookrightarrow [\mathbf{A}^1/\mathbf{G}_m]^I$, so $\mathcal{X}_\ell \to X$ is an isomorphism over the locus where all the $s_i$'s are isomorphisms, and thus the $s_i$'s that are isomorphisms can be ignored in the construction of $\mathcal{X}_\ell$. If $\ell \mid \ell'$, there is an evident $X$-map $\mathcal{X}_{\ell'} \to \mathcal{X}_\ell$. Write $Z_i = Z(s_i)$, $Z = \cup_i Z_i \subset X$, and let $U = X - Z$. Write $\mathcal{Z}_{i,\ell}$ for the preimage of $Z_i$ in $\mathcal{X}_\ell$, and let $\mathcal{Z}_\ell = \cup_i \mathcal{Z}_{i,\ell}$ be the preimage of $Z$. It is sometimes also convenient to generalize this notation slightly: if $J \subset I$ is a non-empty subset, then set $Z_J = \cap_{j \in J} Z_j$, and let $\mathcal{Z}_{J,\ell}$ denote its preimage.
\end{notation}

The universal example of this construction is provided by the following:

\begin{example}
\label{ex:ModelSNC}
Let $X = \mathbf{A}^n$, let $I = \{1,...,n\}$, let $L_i  = \mathcal{O}_X$, and let $s_1,...,s_n$ be the $n$ co-ordinate functions on $X$. Thus, each $Z_i \subset X$ is a co-ordinate hyperplane, and the pair $(X,\{Z_i\}_{i \in I})$ forms the prototypical example of an SNC pair, i.e.,  a regular scheme equipped with a simple normal crossings divisor. In this case, $\mathcal{X}_\ell$ is given by $[\mathbf{A}^n/\mu_\ell^{\oplus n}]$ where the co-ordinate functions are $t_1,...,t_n$, the $\mu_\ell^{\oplus n}$-action is the standard one, and the structure map $\mathcal{X}_\ell \to X$ is given by raising all co-ordinates to their $\ell$-th power. Thus, the closed substack $\mathcal{Z}_{i,\ell} \subset \mathcal{X}_\ell$ is given by the vanishing locus of $t_i^\ell$. In particular, the underlying reduced substack $(\mathcal{Z}_{i,\ell})_{red}$ is smooth divisor, and the pair $(\mathcal{X}_\ell, \{(\mathcal{Z}_{i,\ell})_{red}\}_{i \in I})$ is again an SNC pair.
\end{example}

The construction of root stacks behaves well with respect to base change:

\begin{lemma}
\label{lem:RootStackBC}
The formation of $\mathcal{X}_\ell \to X$ from $\{(L_i,s_i)\}_{i \in I}$ commutes with base change on $X$ for each $\ell$.
\end{lemma}
\begin{proof}
This is immediate from the definition.
\end{proof}

Root stacks do not modify the target scheme $X$ away from the zero locus of the sections $s_i$:

\begin{lemma}
\label{lem:RootStackTrivial}
Each $\mathcal{X}_\ell \to X$ is an isomorphism over $U$.
\end{lemma}
\begin{proof}
This was already observed  above in Notation~\ref{not:RootStack}.
\end{proof}

The most important source of examples arises from a simple normal crossings divisor on a regular scheme, i.e., objects modelled on those in Example~\ref{ex:ModelSNC}.  In this case the line bundles are the ideal sheaves of the irreducible components of the divisor with the natural maps to the structure sheaf. One has:

\begin{lemma}
\label{lem:RootStackSNCD}
Assume that $(X, \{ Z_i\}_{i \in I})$ is an SNC pair, i.e., $X$ is regular, each $Z_i \subset X$ is an effective Cartier divisor that is regular, each $Z_J$ has codimension $\#J$ for non-empty subsets $J \subset I$, and the divisor $Z = \sum_i Z_i$ is reduced with simple normal crossings. Then
\begin{enumerate}
\item For each $\ell$, the pair $(\mathcal{X}_\ell, \{ (\mathcal{Z}_{i,\ell})_{red} \}_{i \in I}  )   $ is a SNC pair. 
\item (Abhyankar) The restriction functor
\[ \mathrm{Fet}^{(p)}(\mathcal{X}_\infty) := \colim_\ell \mathrm{Fet}^{(p)}(\mathcal{X}_\ell) \to \mathrm{Fet}^{(p)}(U) \]
is an equivalence.
%\item For each non-empty subset $J \subset I$, the stack $(\mathcal{Z}_{J,\ell})_{red}$ is regular for each $\ell$. 
%\item The closed substack $(\mathcal{Z}_\ell)_{red} \subset \mathcal{X}_\ell$ is a simple normal crossings divisor with components $(\mathcal{Z}_{i,\ell})_{red}$ for each $\ell$.
\item For each $\ell \geq 0$, the canonical maps induce an equivalence
\[ \xymatrix{ \mathrm{Fet}(\mathcal{Z}_\ell) \ar[r]^-{\simeq} &  \lim \Big( \prod_{i \in I} \mathrm{Fet}(\mathcal{Z}_{i,\ell}) \ar@<0.5ex>[r] \ar@<-0.5ex>[r] & \prod_{i,j \in I} \mathrm{Fet}(\mathcal{Z}_{ \{i,j\}, \ell }) \ar@<0.6ex>[r] \ar@<-0.6ex>[r] \ar[r] & \prod_{i,j,k \in I} \mathrm{Fet}(\mathcal{Z}_{\{i,j,k\},\ell}) \Big),   }\]
and similarly for $\mathrm{Fet}(-)$ replaced by $\mathrm{Fet}^{(p)}(-)$. 
\end{enumerate}
\end{lemma}
\begin{proof}
Statement (1) follows from \cite[\S XIII, Proposition 5.1]{SGA1}, while (2) results from \cite[\S XIII, Proposition 5.2]{SGA1} by passing to quotients. Finally, (3) is a consequence of $h$-descent for finite \'etale covers (i.e., Lemma~\ref{lem:FetCD}) applied to the $h$-cover $\sqcup_i \mathcal{Z}_{i,\ell} \to \mathcal{Z}_\ell$.
%Fix a point $x \in X$. Let $J \subset I$ be subset spanned by those $i \in I$ such that $x \in Z_i$. Working locally near $x$, we may assume $J = I$: the rest of the indices determine map $X \to [\mathbf{G}_m/\mathbf{G}_m]$, and can thus be ignored in the construction of the stacks $\mathcal{X}_\ell$. Our assumptions imply that $\mathcal{O}_{X,x}$ is regular, and that for each $j \in Z$, the divisor $Z_j$ is defined by some $f_j \in \mathcal{O}_{X,x}$ such that $\{f_j\}_{i \in J}$ spans an $\#J$-dimensional vector space $\mathfrak{m}_x/\mathfrak{m}_x^2$. Consider the pushout 
%\[ \xymatrix{ \mathbf{Z}[\{x_j\}_{j \in J}]  \ar[r]^-{x_j \mapsto f_j} \ar[d]^-{x_j \mapsto y_j^\ell} &  \mathcal{O}_{X,x} \ar[d] \\
% \mathbf{Z}[\{y_j\}_{j \in J}] \ar[r] & R.}\]
% Then $R$ describes the semilocal ring of $\mathcal{X}_\ell$ over $x$. It is then easy to see that $R$ is regular and that the set $\{y_j\}_{j \in J}$ spans a $\#J$-dimensional subspace of cotangent space of $\mathrm{Spec}(R)$ at any point above $x$. Now, given $j \in J$, the closed substack $\mathcal{Z}_{j,\ell} \subset \mathcal{X}_\ell$ is defined in $R$ by the equation $f_j$, so $(\mathcal{Z}_{j,\ell})_{red}$ is defined by $y_j$. One can then immediately check (1). 
 \end{proof}

\begin{remark}
Concretely, Lemma~\ref{lem:RootStackSNCD} (3) asserts the following: to specify a finite \'etale cover on $\mathcal{Z}_\ell$, we must specify a finite \'etale cover on each irreducible component $\mathcal{Z}_{i,\ell}$ for $i \in I$, together with a transitive system of isomorphisms over the intersections $\mathcal{Z}_{i,\ell} \cap \mathcal{Z}_{j,\ell}$.
\end{remark}

We give an explicit example of the phenomenon being encoded in Lemma~\ref{lem:RootStackSNCD} globally.

\begin{example}
Let $Y$ be the henselization of $\mathbf{A}^2$ at $0$ over an algebraically closed field $k$ of characteristic $0$, and let $X$ be the blowup of $Y$ at $0$. Let $E \subset X$ be the exceptional divisor, so $E \simeq \mathbf{P}^1$, and $(X,E)$ is an SNC pair. Performing the root stack construction with respect to $E$ gives stacks $\mathcal{X}_\ell \to X$ for each $\ell \geq 1$. Let $\mathcal{E}_\ell \to E$ be the maximal reduced closed substack of $\mathcal{X}_\ell \times _XZ$. Let $\mathcal{E}_\infty := \lim_\ell \mathcal{E}_\ell$. Lemma~\ref{lem:RootStackSNCD} (2) and Lemma~\ref{lem:FetProperHenselian} imply that 
\[ \pi_1(\mathcal{E}_\infty) \simeq \pi_1(X-E) = \pi_1(Y - \{0\}) = 0.\]
We explain how to see this explicitly. The map $\mathcal{E}_\infty \to E$ is naturally a $\widehat{\mathbf{Z}}(1)$-gerbe. Moreover, this gerbe is non-trivial: its class in $H^2(E, \widehat{\mathbf{Z}}(1))$ is the first Chern class of $\mathcal{O}_E(1)$ (which is the conormal bundle of $E \subset X$). One can then show the following general assertion:  for a $\widehat{\mathbf{Z}}(1)$-gerbe $\mathcal{Z} \to Z$ over a simply connected normal base $Z$, the group $\pi_1(\mathcal{Z})$ identifies with the dual of the kernel of the map $\widehat{\mathbf{Z}} \to H^2(Z, \widehat{\mathbf{Z}}(1))$ classifying the gerbe. For example, when $k = \mathbf{C}$, one may use the long exact sequence on homotopy groups to arrive at this conclusion. In our situation, as the conormal bundle $\mathcal{O}_E(1)$ is ample, the corresponding map $\widehat{\mathbf{Z}} \to H^2(E, \widehat{\mathbf{Z}}(1))$ is injective, and thus $\pi_1(\mathcal{E}_\infty) \simeq 0$.
\end{example}

Using Abhyankar's lemma, one proves the invariance of prime-to-$p$ \'etale fundamental groups in the fibers of a proper smooth families. In fact, as we explain, $h$-descent techniques allow us to extend this invariance to certain ``equisingular'' families.

\begin{proposition}
\label{prop:AbhyankarNotSmooth}
Let $k$ be a $\mathbf{Z}_{(p)}$-field,  let $V$ be a henselian dvr over $k$, and let $X$ be a $V$-scheme. Assume that \'etale locally on $X$ we can find an \'etale map $X \to X_0 \otimes_k V$ for a finite type $k$-scheme $X_0$. Let $\eta \rightsquigarrow s$ be a specialization of geometric points on $\mathrm{Spec}(V)$, and write $\overline{V}$ for the integral closure of $V$ in $\eta$. Then the generic fibre functor induces an equivalence
\[\mathrm{Fet}^{(p)}(X_{\overline{V}}) \simeq \mathrm{Fet}^{(p)}(X_\eta).\]
If $X$ is proper over $V$, then the specialization map gives an equivalence
\[ \mathrm{Fet}^{(p)}(X_s) \simeq \mathrm{Fet}^{(p)}(X_\eta).\]
\end{proposition}
\begin{proof}
The second part follows from the first part using Lemma~\ref{lem:FetProperHenselian}. 

For the first part, assume first that $X$ is $V$-smooth. Then $X_\eta \to X_{\overline{V}}$ is a dense open immersion of normal schemes, so $\mathrm{Fet}(X_{\overline{V}}) \to \mathrm{Fet}(X_\eta)$ is fully faithful, and so the same holds for the prime-to-$p$ variant. It remains to check that any prime-to-$p$ finite \'etale cover $Y \to X_\eta$ lifts to a (necessarily prime-to-$p$) finite \'etale cover of $X_{\overline{V}}$. By finite presentation and possibly making $V$ larger, the cover $Y \to X_\eta$ arises as the base change of a prime-to-$p$ finite \'etale cover $Y_0 \to X_{\underline{\eta}}$, where $\underline{\eta} \in \mathrm{Spec}(V)$ denotes the generic point. Let $t \in V$ be a uniformizer. By Abhyankar's lemma (Lemma~\ref{lem:RootStackSNCD} (ii)) applied to $X$ and the divisor $Z(t) \subset X$, it follows that $Y_0 \to X_{\underline{\eta}}$ lifts to a finite \'etale cover of $X$, at least after replacing $X$ with $X \otimes_V V_\ell$, where $V_\ell := V[t^{\frac{1}{\ell}}]$ and $\ell$ is a large integer prime-to-$p$. Base changing along $V_\ell \to \overline{V}$ (determined by choosing an $\ell$-th root of $t$ in $\overline{V}$) then solves our problem.

For general $X$, note that the assertion of the proposition is \'etale local on $X$. Thus, we may choose an \'etale map $X \to X_0 \otimes_k V$ for a finite type $k$-scheme $X_0$. We now argue using hypercovers.  Choose a $2$-truncated $h$-hypercover $Y_{\bullet , 0} \to X_0$ with each $Y_{i, 0}$ being smooth over $k$. By base change, this gives a $2$-truncated $h$-hypercover $Y_\bullet \to X$ with each $Y_i$ being smooth over $V$. Base changing to $\overline{V}$, we have a commutative diagram of pullback maps
\[ \xymatrix{ \mathrm{Fet}^{(p)}(X_{\overline{V}}) \ar[r] \ar[d] & \lim \mathrm{Fet}^{(p)}(Y_{\bullet , \overline{V}}) \ar[d] \\
		 \mathrm{Fet}^{(p)} (X_\eta)\ar[r] & \lim \mathrm{Fet}{(p)}(Y_{\bullet , \eta})^. }\]
The horizontal maps are equivalences by Lemma~\ref{lem:FetCD}, while the right vertical map is an equivalence by the smooth case treated above. It follows that the left vertical map is also an equivalence, as wanted.
\end{proof}

The next proposition is the logarithmic version of the fact that the (prime-to-$p$) \'etale fundamental group behaves constructibly in a proper family.

\begin{proposition}
\label{prop:Logpi1Constructibility}
Let $f:X \to S$ be a proper morphism of finite type $k$-schemes, where $k$ is an algebraically closed field over $\mathbf{Z}_{(p)}$, and $X$ is as in Notation~\ref{not:RootStack}. Then there exists a stratification $\{S_\lambda\}_{\lambda \in \Lambda}$ of $S$ such that if $\eta \rightsquigarrow s$ is any specialization of geometric points of some fixed $S_\lambda$, then the specialization map induces an isomorphism 
\[ \mathrm{Fet}^{(p)}(\mathcal{X}_{\ell,s}) \simeq \mathrm{Fet}^{(p)}(\mathcal{X}_{\ell,\eta})\]
for each $\ell \geq 0$.
\end{proposition}

The crucial assertion above is that the stratification $\{S_\lambda\}_{\lambda \in \Lambda}$ is independent of $\ell$, and thus implies that the association $s \mapsto \colim\limits_\ell \mathrm{Fet}(\mathcal{X}_{\ell,s})$ behaves constructibly in the geometric point $s \to S$.

\begin{proof}
As the formation of $\mathcal{X}_\ell$ is compatible with base change on $X$, it suffices to find an open stratum, i.e., we are free to replace $S$ by a dense open subset by noetherian induction. Moreover, as specializations of geometric points lift along $h$-covers, we may also replace $S$ with $h$-covers (such as finite covers). 

By Lemma~\ref{lem:FetCD}, we may shrink $S$ and replace $X$ by the connected components of the terms of a truncated hypercover to reduce to the following special case: the scheme $S$ is smooth and connected, the map $f:X \to S$ is a smooth projective morphism with connected fibres, and each $Z_i$ is either effective Cartier divisor (when $s_i \neq 0$) or all of $X$ (when $s_i = 0$). Write $I = I_0 \sqcup I_1$, where $i \in I$ lies in $I_0$ if and only if $s_i = 0$. Passing to further $h$-covers of $X$, we may also assume that $W := \cup_{i \in I_1} Z_i$ is the support of a strict normal crossings divisor with irreducible components $D_j$ for $j$ in some finite index set $J$, and that each $Z_i \subset W$ is the support of a subdivisor for $i \in I_1$. Thus, for $i \in I_1$, on the geometric generic fibre $X_\eta$ of $f$, we have $Z_{i,\eta} = \sum_j \delta_{ij} D_{j,\eta}$ for suitable $\delta_{ij} \geq 0$. After possibly shrinking $S$ and passing to a finite cover, we may assume that the same holds globally on $X$, as well as on any geometric fibre $X_s$ of $X \to S$.

In this situation, fix specialization $\eta \rightsquigarrow s$ of geometric points of $S$ witnessed by a map $\mathrm{Spec}(V) \to S$ with $V$ being the absolute integral closure of a henselian dvr.  For each integer $\ell \geq 0$ that is prime-to-$p$, let $\mathcal{X}_{\ell,V} \to X_V \to\mathrm{Spec}(V)$ be the base change of $\mathcal{X}_\ell \to X \to S$ to $V$, and let $\mathcal{X}_{\ell,\eta} \to X_\eta \to \eta$ be its generic fibre. It suffices to show that pullback induces equivalences
\begin{equation}
\label{eq:RootStackSpecialize}
\mathrm{Fet}^{(p)}(\mathcal{X}_{\ell,\eta}) \xleftarrow{\simeq} \mathrm{Fet}^{(p)} (\mathcal{X}_{\ell,V})\xrightarrow{\simeq} \mathrm{Fet}^{(p)}(\mathcal{X}_{\ell,s})
\end{equation}
for each $\ell \geq 0$. In fact, Lemma~\ref{lem:FetProperHenselian} immediately implies the claim for the second map, so it suffices to show that the generic fibre functor induces an equivalence in the first map above.

We may now work \'etale locally on $X_V$. Thus, we may assume $X_V := \mathrm{Spec}(R)$ is affine and \'etale over $V[x_1,...,x_N]$, the index set $J$ of the divisors is $\{1,...,N\}$ with $D_j$ being $Z(x_j) \subset X_V$; here we allow ourselves to expand the index set $J$ of the divisors $D_j$ if necessary to include each $Z(x_j)$. The sections $s_i$ are then either $0$ (corresponding to $i \in I_0$) or of the form $u_i \prod_{j=1}^N x_j^{\delta_{ij}}$ where $u_i \in R^*$ (corresponding to $i \in I_1$). As the formation of $\mathcal{X}_{\ell,V}$ is insensitive to scaling the $s_i$'s by units, we may assume $u_i = 1$ for each $i \in I_1$. The stack $\mathcal{X}_{\ell,V}$ is then $[\mathrm{Spec}(R')/(\mu_\ell)^{\oplus n}]$, where 
\[ R' := R[z_1,...,z_n]/ \Big( \{z_i^\ell\}_{i \in I_0}, \{z_i^\ell - \prod_j x_j^{\delta_{ij}}\}_{i \in I_1} \Big),\]
with $(\mu_\ell)^{\oplus n}$-action given by $(\zeta_i) \cdot (z_i) := (\zeta_i \cdot z_i)$. But then $R'$ is \'etale over 
\[ \Big(k[x_1,...,x_N,z_1,...,z_n] / \big( \{z_i^\ell\}_{i \in I_0}, \{z_i^\ell - \prod_j x_j^{\delta_{ij}}\}_{i \in I_1} \big)\Big) \otimes_k V, \]
so Proposition~\ref{prop:AbhyankarNotSmooth} shows that the generic fibre functor gives an equivalence
\[ \mathrm{Fet}^{(p)}(R') \simeq \mathrm{Fet}^{(p)}(R'_\eta).\]
As this equivalence is functorial, it is compatible with the $(\mu_\ell)^{\oplus n}$-action, so we obtain the desired equivalence in \eqref{eq:RootStackSpecialize} by applying the functor of (homotopy) $(\mu_\ell)^{\oplus n}$-fixed points to the previous equivalence.
\end{proof}

%\newpage

\section{Fundamental groups of punctured tubular neighbourhoods}
\label{S:section3}

This section is only relevant for the proof of Theorem~\ref{thm:GKPIntro}, and can be skipped for readers only interested in the proof of Theorem~\ref{thm:XuIntro} by characteristic $p$ methods. We fix an algebraically closed field $k$ of characteristic $p$ (possibly $0$). Let $X$ be a finite type separated $k$-scheme,  and let $i:Z \hookrightarrow X$ be a closed subscheme; write $j:U \hookrightarrow X$ for its complement. Our goal in this section is to show that the topology of $U$ near a point $z \in Z$ behaves constructibly in $z$. To make this precise, we introduce some notation concerning ``tubes'' in $U$ of points of $Z$.

\begin{notation}
For a geometric point $z \to Z$, we write $T_z(Z)$ for the punctured tubular neighbourhood of $z$ in $X$, i.e., $T_z(Z)$ is the preimage of $U$ under $X^{sh}_z \to X$; note that this definition only requires $Z$ to be closed in $X$ in a neighbourhood of $z$, and thus makes sense for $Z$ assumed be only locally closed.  A specialization $\eta \rightsquigarrow s$ of geometric points of $Z$ defines a map $X_\eta^{sh} \to X_s^{sh}$, and hence a map $T_\eta(Z) \to T_s(Z)$, which we call the specialization map.
%For any map $f:z \to Z$, write $X_z^{sh}$ for the henselization of $X$ along $z$ (i.e., the inverse limit of all \'etale neighbourhoods of $f$), and let $T_z \subset X_z^{sh}$ for the complement of the inverse image of $U$ in $X_z^{sh}$. In fact, this notation will be used only when $z$ is either a geometric point, or of the form $\mathrm{Spec}(V)$ for an absolutely integrally closed valuation ring $V$. In the latter case, if $\eta$ and $s$ denote the generic and closed points of $\mathrm{Spec}(V)$, then \'etale neighbourhoods of $\mathrm{Spec}(V) \to X$ and $z \to X$ coincide, so we have $T_V \simeq T_s$, which gives an induced map $T_\eta \to T_V \simeq T_s$; we thus say the map $\mathrm{Spec}(V) \to Z$ withnesses the specialization $\eta \rightsquigarrow s$ of geometric points of $Z$. 
\end{notation}

The main result of this section is:

\begin{theorem}
\label{thm:ConstructibilityMilnorpi1}
There exists a stratification $\{Z_\lambda\}_{\lambda \in \Lambda}$ of $Z$ with the following property: for a fixed stratum $Z_\lambda$ and for any specialization $\eta \rightsquigarrow s$ of geometric points of $Z_\lambda$, the specialization map induces an equivalence
\[ \mathrm{Fet}^{(p)} (T_s(Z))\simeq \mathrm{Fet}^{(p)}(T_\eta(Z)).\]
In particular, if $X$ is geometrically unibranch and the $Z_\lambda$ are chosen to be connected (which is always possible after refining the stratification further), then we get isomorphisms
\[ \pi_1(T_{z_1}(Z))^{(p)} \simeq \pi_1(T_{z_2}(Z))^{(p)}\]
for any two geometric points $z_1,z_2$ of $Z_i$ (depending on the choice an \'etale path linking $z_1$ and $z_2$). 
\end{theorem}

The main idea of the proof is to use hypercovers of $(X,Z)$ by SNC pairs to describe $\pi_1(T_z(Z))^{(p)}$ in terms of the fiber over $z$ of a certain (finite) diagram of proper schemes over $Z$; the construction of this diagram is independent of the point $z \to Z$, so we can then conclude by invoking the general constructibility for the fundamental groups of the fibers of certain proper maps of  schemes proven in Proposition~\ref{prop:Logpi1Constructibility}. 

\begin{remark} The proper setting for these results is log geometry in the sense of Fontaine, Illusie, and Kato, but to minimize notational complications, we avoid log schemes in the proof below, and work instead with the equivalent formulation in terms of root stacks, as in \S \ref{ss:RootStacks}.
\end{remark}

\begin{proof}[Proof of Theorem~\ref{thm:ConstructibilityMilnorpi1}]
Choose a truncated hypercover $\pi:X_\bullet \to X$ such that each $X_i$ is smooth, and $\pi^{-1}(Z)_{red} =: Z_\bullet \subset X_\bullet$ is a simple normal crossings divisor. Let $U_\bullet := \pi^{-1}(U)$. Associated to the divisors in $Z_\bullet$, we have a tower of root stacks $\mathcal{X}_{\bullet , \ell} \to X_\bullet$ indexed by positive integers $\ell \in \mathbf{Z}$ that are coprime to $p$, as in Notation~\ref{not:RootStack}; the restriction of this root stack to $U_\bullet$ is an isomorphism by Lemma~\ref{lem:RootStackTrivial}, while the restriction $\mathcal{Z}_{\bullet, \ell} \to Z_\bullet$ to $Z_\bullet$ identifies with the corresponding infinite root stack for the restricted divisors (viewed as pairs comprising of a line bundle and a section, possibly $0$) by Lemma~\ref{lem:RootStackBC}. This data fits into the following diagram:
\[ \xymatrix{ U_\bullet \ar[r] \ar@{=}[d] & \mathcal{X}_{\bullet , \ell} \ar[d] & \mathcal{Z}_{\bullet , \ell} \ar[d] \ar[l] \\
		 U_\bullet \ar[r] \ar[d] & X_\bullet \ar[d] & Z_\bullet \ar[d] \ar[l] \\
		 	U \ar[r] & X & Z. \ar[l]}\]
Fix a geometric point $z \to Z$, and consider the base change of the above diagram along $X_z^{sh} \to X$. This gives
\[ \xymatrix{ U_\bullet \times_U T_z(Z) \ar[r] \ar@{=}[d] & \mathcal{X}_{\bullet , \ell} \times_X X^{sh}_z \ar[d] & \mathcal{Z}_{\bullet , \ell}\times_Z Z^{sh}_z \ar[d] \ar[l] \\
		 U_\bullet  \times_U T_z(Z) \ar[r] \ar[d] & X_\bullet \times_X X^{sh}_z \ar[d] & Z_\bullet \times_Z Z^{sh}_z \ar[d] \ar[l] \\
		 	T_z(Z) \ar[r] & X^{sh}_z & Z^{sh}_z. \ar[l]}\]
Lemma~\ref{lem:FetCD} gives an equivalence
\[ \mathrm{Fet}(T_z(Z)) \simeq \lim \mathrm{Fet}(U_\bullet \times_U T_z(Z)),\]
and similarly for the prime-to-$p$ variant. Lemma~\ref{lem:RootStackSNCD} (2) gives an equivalence
\[ \mathrm{Fet}^{(p)}(U_\bullet \times_U T_z(Z)) \simeq \colim\limits_\ell \mathrm{Fet}^{(p)}(\mathcal{X}_{\bullet , \ell} \times_X X^{sh}_z).\]
Lemma~\ref{lem:FetProperHenselian} then shows
\[ \mathrm{Fet}(\mathcal{X}_{\bullet , \ell} \times_X X^{sh}_z) \simeq \mathrm{Fet}(\mathcal{X}_{\bullet , \ell, z}) \simeq \mathrm{Fet}(\mathcal{Z}_{\bullet,\ell,z}),\]
and similarly for the prime-to-$p$ variant. Combining these, we learn that 
\[ \mathrm{Fet}^{(p)}(T_z(Z)) \simeq \lim \colim_\ell \mathrm{Fet}^{(p)}(\mathcal{Z}_{\bullet , \ell,z}).\]
Applying Proposition~\ref{prop:Logpi1Constructibility} to each term of the map $Z^i \to Z$ (with line bundles and sections as above) then proves the desired constructibility.
\end{proof}

The following variant of Theorem~\ref{thm:ConstructibilityMilnorpi1} concerns slightly larger ``tubes'' arounds points of $Z$, and is recorded here convenience of reference. 

%As in section \ref{S:section3} let $k$ be an algebraically closed field of characteristic $p$, and let $X$ be a finite type separated $k$-scheme.  For a closed subscheme $W\hookrightarrow X$ and a geometric point $z\rightarrow W$ let $T_z(W)$ denote the preimage of $X-W$ under the map $X_z^{sh}\rightarrow X$.  

\begin{theorem}\label{T:variant} 
There exists a stratification $\{Z_\lambda \}_{\lambda \in \Lambda }$ of $Z$ with the following property: for a fixed stratum $Z_\lambda$ and for any specialization $\eta \rightsquigarrow s$ of geometric points of $Z_\lambda$, the specialization map induces an equivalence
\[ \mathrm{Fet}^{(p)}(T_s(Z_\lambda )) \simeq \mathrm{Fet}^{(p)}(T_\eta(Z_\lambda )).\]
In particular, if $X$ is geometrically unibranch and the $Z_\lambda$ are chosen to be connected (which is always possible after refining the stratification further), then we get isomorphisms
\[ \pi_1(T_{z_1}(Z_\lambda ))^{(p)} \simeq \pi_1(T_{z_2}(Z_\lambda ))^{(p)}\]
for any two geometric points $z_1,z_2$ of $Z_\lambda $ (depending on the choice an \'etale path linking $z_1$ and $z_2$).
\end{theorem}
\begin{proof}
We proceed by induction on the dimension of $Z$.  The base case $\text{dim}(Z) = 0$ is immediate as there are no specializations in $Z$. For the inductive step, it suffices to find a dense open subset $Q\subset Z$ with the following property: for every specialization $\eta \rightsquigarrow s$ of geometric points of $Q$ the specialization map induces an equivalence
\[ \mathrm{Fet}^{(p)}(T_s(Q)) \simeq \mathrm{Fet}^{(p)}(T_\eta(Q )).\]
But then we may simply take $Q$ to be the dense open stratum stratum of $Z$ provided by Theorem~\ref{thm:ConstructibilityMilnorpi1}.
%For this let $Q$ be the intersection of the open stratum of $Z$ provided by \ref{thm:ConstructibilityMilnorpi1} with the complement in $Z$ of the intersections of irreducible components of $Z$.
\end{proof}

\begin{remark}
\label{rmk:ConstructibilityCohomology}
The arguments of this section can also be adapted to show that the prime-to-$p$ \'etale homotopy type of $T_z(Z)$ behaves constructibly in $z$.  This is not used in what follows, however, and we do not discuss it further.
 %More precisely, fix a standard coefficient ring $A \in \{ \mathbf{F}_\ell, \mathbf{Z}/\ell^n, \mathbf{Z}_\ell, \mathbf{Q}_\ell, \overline{\mathbf{Q}_\ell} \}$. One can then find a stratification $\{Z_\lambda\}_{\lambda \in \Lambda}$ of $Z$ satisfying the conclusion of  Theorem~\ref{thm:ConstructibilityMilnorpi1}, as well as the following: given a specialization $\eta \rightsquigarrow s$ of geometric points in some $Z_\lambda$ and an $A$-local system $\mathcal{F}$ on $T_s$, the specialization map $sp:T_\eta \to T_s$ induces an isomorphism $R\Gamma(T_s, \mathcal{F}) \to R\Gamma(T_\eta, sp^* \mathcal{F})$.
\end{remark}

%\newpage

\section{Local fundamental groups in arithmetic families}\label{S:section4}

The goal of this section is to show that local fundamental groups of singularities behave constructibly in arithmetic families. In particular, prime-to-$p$ quotients of local fundamental groups of singularities in characteristic $0$ can be calculated using characteristic $p$ methods for most $p \gg 0$. The precise result is:

\begin{theorem}\label{thm:constructibility}
Let $A$ be a finitely generated $\mathbf{Z}$-algebra equipped with an embedding $A \hookrightarrow \mathbf{C}$. Fix a pointed finite type $A$-scheme $(Z_A,z_A)$, and let $(Z,z)$ denote the corresponding pointed scheme over $\mathbf{C}$. There exists a dense open $U \subset \mathrm{Spec}(A)$ with the following property: for every map $\mathrm{Spec}(\kappa) \to U$ with $\kappa$ an algebraically closed field of characteristic $p$, we have a canonical equivalence
\[ \mathrm{Fet}^{(p)}(Z_z^{sh, \circ}) \simeq \mathrm{Fet}^{(p)}(Z^{sh, \circ}_{\kappa,z_{\kappa}}),\]
where $(Z_\kappa,z_\kappa)$ is the base change of $(Z_A,z_A)$ along $A \to \kappa$. If $Z$ is geometrically unibranch near $z$, then after possibly shrinking on $U$ we can further arrange that $Z_\kappa $ is unibranch near $z_\kappa $ for all $\mathrm{Spec}(\kappa )\rightarrow U$, in which case the  previous conclusion can be interpreted as an isomorphism
\[ \pi_1^{loc}(Z,z)^{(p)} \simeq \pi_1^{loc}(Z_\kappa,z_\kappa)^{(p)}.\]
\end{theorem}

The main idea of the proof is to use $h$-hypercovers by simple normal crossings pairs to describe $\mathrm{Fet}^{(p)}(Z_z^{sh,\circ})$ in terms of certain related smooth proper schemes over $z$, similar to what was done in the proof of Theorem~\ref{thm:ConstructibilityMilnorpi1}. This description is compatible with ``spreading out'' over $\mathrm{Spec}(A)$, so we can then end by invoking the invariance of $\mathrm{Fet}^{(p)}(-)$ under specialization to characteristic $p$ for smooth proper varieties.

\begin{proof}
Using resolution of singularities (or alterations), we may choose a truncated proper hypercover $f:X_\bullet \to Z$ indexed by $\bullet \in \Delta_{\leq 2}^{\mathrm{op}}$ with $X_i$ smooth, and $D_\bullet := f^{-1}(z)_{red} \subset X_\bullet$ giving an SNC divisor at each level; let $I_\bullet$ be the finite index set of components of $D_\bullet$, so, for any $i \in I_\bullet$, the component $D_{\bullet, i}$ is a smooth proper variety over $\mathbf{C}$. Let $U_\bullet = X_\bullet - D_\bullet = f^{-1}(Z - \{z\})$ be the complement. Let $\mathcal{X}_{\bullet , \ell} \to X_\bullet$ be the $\ell$-th root stack associated to the divisors in $D_\bullet$, and let $\mathcal{D}_{\bullet , \ell} \to D_\bullet$ be its pullback to $D$. Finally, we base change everything along $Z_z^{sh} \to Z$, so $U^{sh}_{\bullet , z} := U_\bullet \times_Z Z^{sh}_z$, etc.. Then we have canonical equivalences
\[ \mathrm{Fet}(Z_z^{sh,\circ}) \xrightarrow{\simeq} \lim \mathrm{Fet}(U^{sh}_{\bullet, z}) \xleftarrow{\simeq} \lim \colim\limits_\ell \mathrm{Fet}(\mathcal{X}_{\bullet , \ell,z}^{sh}) \xrightarrow{\simeq} \lim \colim\limits_\ell \mathrm{Fet}(\mathcal{D}_{\bullet, \ell}) \simeq \colim_\ell \lim \mathrm{Fet}(\mathcal{D}_{\bullet, \ell})\]
induced by the relevant pullback maps. Here the limit takes place over $\bullet \in \Delta_{\leq 2}$, the first isomorphism arises from Lemma~\ref{lem:FetCD}, the second from Lemma~\ref{lem:RootStackSNCD} (ii),  the third from Lemma~\ref{lem:FetProperHenselian}, and the last from the commutation of filtered colimits with finite limits. Using Lemma~\ref{lem:RootStackSNCD} (iii) to simplify the last term, we can describe $\mathrm{Fet}(Z_z^{sh,\circ})$ as
\[
\xymatrix{ \colim\limits_\ell \lim \Big( \prod_{i \in I_\bullet} \mathrm{Fet}(\mathcal{D}_{\bullet, i,\ell}) \ar@<0.5ex>[r] \ar@<-0.5ex>[r] & \prod_{i,j \in I_\bullet} \mathrm{Fet}(\mathcal{D}_{ \bullet ,\{i,j\}, \ell }) \ar@<0.6ex>[r] \ar@<-0.6ex>[r] \ar[r] & \prod_{i,j,k \in I_\bullet} \mathrm{Fet}(\mathcal{D}_{\bullet, \{i,j,k\},\ell}) \Big),   }
\]
where the limit is now over $\Delta_{\leq 2} \times \Delta_{\leq 2}$ and is computed by first limiting over the displayed $\Delta_{\leq 2}$-diagram of categories while holding $\bullet$ fixed, and then limiting over $\bullet \in \Delta_{\leq 2}$. In particular, this gives a description of the category of finite \'etale covers over $Z_z^{sh,\circ}$ in terms of the categories of finite \'etale covers over the smooth proper stacks $(D_{\bullet, J,\ell})_{red}$ for $J \subset I_\bullet$ with $1 \leq \#J \leq 3$, and the same also holds if we use $\mathrm{Fet}^{(p)}(-)$ instead of $\mathrm{Fet}(-)$ for any fixed prime $p$.  

We now specialize. By possibly changing our choice of $f$ and shrinking $\mathrm{Spec}(A)$ if necessary, we can spread $f$ out to a truncated $h$-hypercover $f_A:X_{\bullet , A} \to Z_A$ with $X_{i, A}$ smooth over $A$, and $D_{\bullet ,A} := f^{-1}(z_A)_{red} \subset X_{\bullet , A}$ giving an SNC divisor in $X_{\bullet, A}$ in every fibre over $\mathrm{Spec}(A)$ with components still indexed by the same $I_\bullet$ as in the previous paragraph. We may then form the root stacks $\mathcal{D}_{\bullet, A,\ell} \to D_{\bullet , A}$ as above. Fix a geometric point $a_\kappa:\mathrm{Spec}(\kappa) \to \mathrm{Spec}(A)$ with $\kappa$ having characteristic $p$, and replace the subscript $A$ with the subscript $\kappa$ to denote the passage to the fibre above this point. 
% In this situation, shall show that the specialization map\footnote{The relevant specialization map will be constructed in the proof; it may also be defined intrinsically.} gives an equivalence
%\[ \mathrm{Fet}(Z_z^{sh,\circ})^{(p)} \xleftarrow{\simeq} \mathrm{Fet}(Z_{\kappa,z_\kappa}^{sh,\circ})^{(p)}. \]
%In fact, the first equivalence results from the invariance of \'etale fundamental groups under algebraically closed base field extensions (see Lemma~\ref{lem:Localpi1invariance}), so it suffices to handle the second one.
The analysis in the first paragraph above works directly to describe $\mathrm{Fet}^{(p)}(Z_{\kappa,z_{\kappa}}^{sh,\circ})$ as
\[
\xymatrix{ \colim\limits_\ell \lim\limits_{\Delta_{\leq 2} \times \Delta_{\leq 2}} \Big( \prod_{i \in I_\bullet} \mathrm{Fet}^{(p)}(\mathcal{D}_{\bullet, \kappa,i,\ell}) \ar@<0.5ex>[r] \ar@<-0.5ex>[r] & \prod_{i,j \in I_\bullet} \mathrm{Fet}^{(p)}(\mathcal{D}_{\bullet,  \kappa, \{i,j\}, \ell }) \ar@<0.6ex>[r] \ar@<-0.6ex>[r] \ar[r] & \prod_{i,j,k \in I_\bullet} \mathrm{Fet}^{(p)}(\mathcal{D}_{\bullet, \kappa, \{i,j,k\},\ell}) \Big).   }
\]
As specialization maps for smooth proper varieties are functorial, we therefore get a specialization map
\[ \mathrm{Fet}^{(p)}(Z_z^{sh,\circ}) \xleftarrow{} \mathrm{Fet}^{(p)}(Z_{\kappa,z_\kappa}^{sh,\circ}),\]
which we claim is an equivalence.  To verify this last assertion, it is 
 enough to check that the speciaization map gives an equivalence
\[ \mathrm{Fet}^{(p)}(\mathcal{D}_{\bullet, J,\ell}) \xleftarrow{\simeq} \mathrm{Fet}^{(p)}(\mathcal{D}_{\bullet, \kappa,J,\ell}) \]
for any non-empty subset $J$. This follows from Proposition~\ref{prop:AbhyankarNotSmooth} as $\big(\mathcal{D}_{\bullet, A,J,\ell}\big)_{red}$ is smooth and proper over $A$ for any non-empty subset $J$ by Lemma~\ref{lem:RootStackSNCD} (i). 
\end{proof}

The following proposition is implicitly proven above.

\begin{proposition}
\label{prop:Localpi1invariance}
Let $X$ be a scheme of finite type over an algebraically closed field $K$ of characteristic $p \geq 0$. Let $x \to X$ be a geometric point, and let $L/K$ be be an extension of algebraically closed fields. Choose a geometric point $x_L \to X_L$ lying above $x$. Then pullback induces an equivalence
\[ \mathrm{Fet}^{(p)}(X_x^{sh,\circ}) \simeq \mathrm{Fet}^{(p)}(X_{L,x_L}^{sh,\circ}).\]
\end{proposition}
\begin{proof}
Let $\underline{x} \in X$ denotes the scheme-theoretic point defined by $x$, and similarly for $\underline{x_L} \in X_L$, so $\underline{x}_L$ is the unique point of $X_L$ lifting $\underline{x}$ (where we have uniqueness as base change along an extension of algebraically closed fields preserves irreducibility).  As in the proof of Theorem~\ref{thm:constructibility} or Theorem~\ref{thm:ConstructibilityMilnorpi1}, one has
\[\mathrm{Fet}^{(p)}(X_x^{sh,\circ}) \simeq \lim \mathrm{Fet}^{(p)}(Z_{\bullet, \overline{\kappa(\underline{x})}}),\]
where $Z_\bullet$ is a certain finite diagram of smooth proper stacks over $\kappa(\underline{x})$, and the algebraic closure $\overline{\kappa(\underline{x})}$ is chosen using the geometric point $x$. The formation of this diagram is compatible with base change along $K \to L$, so we are done as 
\[\mathrm{Fet}^{(p)}(Z_{\overline{\kappa(\underline{x})}} ) \simeq \mathrm{Fet}^{(p)}(Z_{\overline{\kappa(\underline{x_L})}})\]
for smooth proper stacks $Z$ over $\kappa(x)$ (essentially by Proposition~\ref{prop:AbhyankarNotSmooth}).
\end{proof}

\section{A criterion for finiteness of a profinite group}
\label{sec:ProfiniteGroup}

In this section, we record the following criterion for the finiteness of a profinite group $G$ in terms of finiteness of its prime-to-$p$ quotients $G^{(p)}$.

\begin{proposition}
\label{prop:ProfiniteGroupFinitenessCriterion}
Let $G$ be a profinite group. Assume that for all but finitely many prime numbers $p$, the group $G^{(p)}$ is finite with order bounded above by a polynomial in $p$. Then $G$ is finite.
\end{proposition}
\begin{proof}
Let $S$ denote the finite set of primes where $G^{(p)}$ is not known to be finite. Fix a positive integer $c$ such that $\#G^{(p)} \leq p^c$ for all $p$ not in $S$.  Assume towards contradiction that $G$ is infinite.

First, as $G$ is infinite, we may choose a tower $\{... \to G_{n+1} \xrightarrow{\psi_n} G_n \to ... \to G_0\}$ of finite quotients of $G$ such that each map $\psi_n$ is surjective with non-trivial kernel. To see this, let $\psi:G \to H$ be a fixed finite quotient of $G$. Then, since $G$ is infinite, there exists another finite quotient $\psi':G \to H'$ that does not factor through $G \to H$. Let $H''$ be the image of the induced map $(\psi,\psi'):G \to H \times H'$. Then $H''$ is a finite quotient of $G$, and the induced projection map $H'' \to H$ is surjective (as $G$ maps onto both compatibly) with non-trivial kernel (since otherwise $\psi'$ would have to factor through $\psi$). Applying this construction inductively gives the required tower $\{G_n\}$. In particular, the sequence $\{\# G_n\}_{n \geq 1}$ of orders is unbounded and satisfies $\#G_n \mid \#G_{n+1}$.

Next, by passing to a cofinal subtower in $\{G_n\}$, we may assume that the number $k(n)$ of distinct prime divisors of $\# G_n$ is also unbounded as $n \to \infty$. Indeed, if not, then the prime divisors of all the $\# G_n$'s are contained in a finite set of primes (as $\#G_n \mid \#G_{n+1}$). Choosing any prime $p$ that is not in this finite set and not in $S$, we see that each $G_n$ is a prime-to-$p$ quotient of $G$. But, since $G^{(p)}$ is finite, the tower $\{G_n\}$ would have be constant, which is a contradiction.

Now observe that the trivial lower bound $k(n)! \leq \# G_n$ then yields, via Stirling's formula, the lower bound
\[ \big(\frac{k(n)}{e}\big)^{k(n)} \leq \# G_n.\]
For each $n$, let $p_n$ be the smallest prime that is not in $S$ and that does not divide $\# G_n$. Then $G_n$ is a prime-to-$p_n$, and hence we obtain
\[ \# G_n \leq  \#G^{(p_n)} \leq p_n^c.\]
By the pigeonhole principle, $p_n$ is amongst the first $g(n) := k(n) + \#S + 1$ primes. Hence, by the prime number theorem, we have
\[ p_n = O\Big(g(n) \log\big(g(n)\big)\Big).\]
Combining this with the above two inequalities for $\#G_n$ gives
\[ \big(\frac{k(n)}{e}\big)^{k(n)} = O\Big(g(n)^c \cdot \log \big(g(n)\big)^c\Big).\]
Now $g(n)$ grows linearly in $k(n)$, so the right side above is $O(k(n)^{2c})$; this clearly contradicts the left side growing exponentially in $k(n)$, so there cannot be a tower $\{G_n\}$ as above, and thus $G$ is finite.
\end{proof}

This result implies the abstract Theorem~\ref{thm:MainThmIntro}:

\begin{proof} [Proof of Theorem~\ref{thm:MainThmIntro}]
Combine Theorem~\ref{thm:constructibility} with Proposition~\ref{prop:ProfiniteGroupFinitenessCriterion}.
\end{proof}

\begin{remark}
The growth bound imposed in Proposition~\ref{prop:ProfiniteGroupFinitenessCriterion} can be weakened (as the proof shows), but not eliminated completely. For example, let $G = \prod_{n \geq 5} A_n$ be the displayed product of alternating groups. We will show that $G^{(p)}$ is finite for all $p$. On the other hand, $G$ is infinite; this does not violate Proposition~\ref{prop:ProfiniteGroupFinitenessCriterion} since (as we will see) $\# G^{(p)} \approx (p-1)!$, which is much larger asymptotically than a polynomial in $p$.

For the finiteness, note that any continuous finite prime-to-$p$ quotient $G \to H$ factors through some finite quotient $G \to G' := \prod_{n = 5}^N A_n$ for $N$ sufficiently large. If $p \in \{2,3\}$, then each induced map $A_n \to G \to H$ must be $0$: the simple group $A_n$ has order divisible by $6$ for $n \geq 4$, while $H$ does not have order divisible by $6$. It follows that $H = 0$, so $G^{(p)} = 0$ for $p \in \{2,3\}$. For $p \geq 5$, the subgroup $G'' := \prod_{n=p}^N A_n \subset G'$ maps trivially to $H$ by a similar argument (as each  summand $A_n \subset G''$ a simple group with order divisible by $p$). It follows that $G \to H$ factors over the projection $G \to \prod_{n=5}^{p-1} A_{n}$, and hence the latter projection is a maximal prime-to-$p$ quotient $G^{(p)}$ of $G$. In particular, $G^{(p)}$ is finite for all $p$, and $\#G^{(p)} = \prod_{n=5}^{p-1} \frac{n!}{2}$ (with the convention that the empty product is $1$).
\end{remark}

\section{Proof of  Theorems \ref{thm:CSTIntro} and \ref{thm:generalbound}}
\label{sec:proofs}

Let $k$ be an algebraically closed field of positive characteristic $p$ and let $X = \mathrm{Spec}(R)$ be a finite type affine $k$-scheme.  Let $x\in X(k)$ be a point.  Assume that $R$ is a normal domain with dimension $d$, and let $R_x$ denote the strict henselization of $X$ at $x$.  Let $\Delta $ be an effective $\mathbf{Q}$-divisor on $X$, and let $\Delta _x$ denote the induced $\mathbf{Q}$-divisor on $\mathrm{Spec}(R_x)$.  We will be concerned with pairs of the following type:

\begin{definition}[{\cite{HochsterHuneke}}]
The pair $(R_x, \Delta _x)$ is called \emph{strongly $F$-regular} if for any nonzero element $c\in R_x$, there exists $e>0$ such that the $R$-linear inclusion
$$
R_x\rightarrow R_x(\lfloor (p^e-1)\Delta \rfloor )^{1/p^e}
$$
sending $1$ to $c^{1/p^e}$ is split. If this condition is merely assumed for $c=1$, we say that $(R_x,\Delta_x)$ is {\em $F$-pure}.  
\end{definition}

 The other main definition we need is the following:
 \begin{definition}[{\cite[Definition 9]{HL}}]\label{D:signature}
 The \emph{$F$-signature} of $R_x$ is defined to be 
$$
s(R_x, x):= \lim _{e\rightarrow \infty }\frac{a_e }{p^{ed}},
$$
where for $e\geq 1$ the integer $a_e$ is defined to be the maximal integer $r$ for which there exists a surjection of $R_x$-modules
$$
\xymatrix{
R_x^{1/p^e}\ar@{->>}[r]& R_x^{\oplus r}}.
$$
\end{definition}

\begin{remark} As discussed in \cite[2.1]{HaraWatanabe}, Definition \ref{D:signature} can be extended to pairs $(X, \Delta )$. The following feature of these definitions seems noteworthy: one does not need to assume that $K_X$, $\Delta$, or $K_X + \Delta$ are $\mathbf{Q}$-Cartier, unlike the corresponding definitions for klt in characteristic $0$.
\end{remark}

The key point to the proof of Theorem \ref{thm:generalbound}, which is  implicit in \cite{CRST:pi1} is the following result.
Let $A\subset \mathbf{C}$ be a finitely generated normal $\mathbf{Z}$-subalgebra and let $(X_A, x_A)$ be a normal pointed flat $A$-scheme.

\begin{proposition}\label{P:bound}  Assume that there exists a $\mathbf{Q}$-divisor $\Delta _A$ on $X_A$ such that the triple $(X_{\mathbf{C}}, x_{\mathbf{C}}, \Delta _{\mathbf{C}})$ obtained by base change along $A\hookrightarrow \mathbf{C}$ has klt singularities  and is of dimension $d$ in a  neighborhood of $x_{\mathbf{C}}$.  Then there exists a dense open subset $U\subset \mathrm{Spec}(A)$ such that for any geometric point $\bar y\rightarrow U$ with positive residue characteristic $p_{\bar y}$ the fiber $(X_{\bar y}, x_{\bar y})$ is strongly $F$-regular and 
$$
s(X_{\bar y}, x_{\bar y})>1/p_{\bar y}^d.
$$
\end{proposition}
\begin{proof}
We may without loss of generality assume that $X_A = \mathrm{Spec}(R)$ is affine.  For a geometric point $\bar y\rightarrow \mathrm{Spec}(A)$ let $R_{\bar y}$ denote the coordinate ring of the geometric fiber of $X_A$.  After shrinking on $\mathrm{Spec}(A)$ we may assume that the fibers are all normal \cite[IV, 12.2.4]{EGA}.
% and the restrictions $\Delta _{\bar y}$ of $\Delta _A$ to the fibers are $\mathbf{Q}$-Cartier.  
After further shrinking on $A$ we may further assume that we have a finite map
$$
P_A:= A[x_1, \dots, x_d]\rightarrow R
$$
Let $c\in P_A$ be a discriminant of this map given by $\text{det}(\text{tr}(e_ie_j))^2$ for a set of generators $e_i$ of $R$ over $P_A$, and 
let $\alpha _0$ denote the log canonical threshold of $(X_{\mathbf{C}}, \Delta _{\mathbf{C}})$ along $\text{div}(c)$ at $x_{\mathbf{C}}$.  So $\alpha _0>0$ and for any $t<\alpha _0$ the pair $(X_{\mathbf{C}}, \Delta _{\mathbf{C}}+t\text{div}(c))$ has log canonical singularities in a neighborhood of $x$ (see \cite[Definition 7.1]{PatakfalviSchwedeTucker}).  Let $\alpha $ be some fixed positive rational number with $\alpha <\alpha _0$.  Then a fundamental theorem of Takagi \cite[Corollary 3.5]{Takagi} (see also \cite[Theorem 7.6]{PatakfalviSchwedeTucker}) ensures that there exists a dense open $U \subset \mathrm{Spec}(A)$ such that for any geometric point $\bar y\rightarrow U$ the pair $(X_{\bar y}, \Delta _{\bar y}+t\text{div}(c_{\bar y}))$ is $F$-pure for $t\leq \alpha $.  As $\frac{1}{p-1}<\alpha $ for all but finitely many primes $p$, it follows that after shrinking on $U$ we can ensure that the pair $(X_{\bar y}, \Delta _{\bar y}+\frac{1}{p_{\bar y}-1}\text{div}(c_{\bar y}))$ is $F$-pure for all $\bar y\rightarrow U$ mapping to a closed point of $U$.  This implies, in particular, that the pair $(X_{\bar y}, \frac{1}{p_{\bar y}-1}\text{div}(c_{\bar y}))$ is $F$-pure and therefore for all $\bar y\rightarrow U$ mapping to a closed point of $U$ the map
$$
\xymatrix{
R_{x_{\bar y}}\ar[r]^-{(-)^{p_{\bar y}}}& R_{x_{\bar y}}^{1/p_{\bar y}}\ar[r]^-{c_{\bar y}^{1/p_{\bar y}}}& R_{x_{\bar y}}^{1/p_{\bar y}}}
$$
admits an $R_{x_{\bar y}}$-splitting.  Now the first proof of \cite[Theorem 5.1]{PolstraTuckerFsig} shows that 
$$
s(X_{\bar y}, x_{\bar y})>1/p_{\bar y}^d
$$
as desired.
\end{proof}

\begin{proof}[Proof of Theorems~\ref{thm:CSTIntro} and \ref{thm:generalbound}]
In the setup of Theorem~\ref{thm:CSTIntro}, we have
$$
\# \pi_1^{loc}(X,x)\leq 1/s(X, x)
$$
by \cite[Theorem 5.1]{CRST:pi1}; this yields Theorem~\ref{thm:CSTIntro}. Combining this with Proposition~\ref{P:bound} we obtain Theorem~\ref{thm:generalbound}.
%To complete the proof of \ref{thm:CSTIntro} we again appeal to \cite[Theorem 5.1]{PolstraTuckerFsig} as follows.  We may assume that $X = \mathrm{Spec}(R)$ is affine. Choose a Noether normalization $P = k[x_1, \dots, x_d]\rightarrow R$, let $c\in P$ be a discriminant, and let $e\geq 1$ be an integer such that the composition
%$$
%\xymatrix{
%R\ar[r]^-{(-)^{p^e}}& R^{1/p^e}\ar[r]^-{c^{1/p^e}}& R^{1/p^e}}
%$$
%admits an $R$-module splitting.  Such an integer $e$ exists because $R$ is strongly $F$-regular.  Then by \cite[First proof of Theorem 5.1]{PolstraTuckerFsig} we have
%$$
%s(X, x)\leq 1/p^{ed}.
%$$
\end{proof}

As promised, we also get a new proof of Xu's theorem:
 
\begin{proof}[Proof of Theorem~\ref{thm:XuIntro}]
Combine Theorem~\ref{thm:MainThmIntro} with Theorem~\ref{thm:generalbound}.
\end{proof}

\section{A result of Greb, Kebekus, and Peternell}
\label{sec:GKP}

The techniques from \S \ref{S:section3} and \S \ref{S:section4} can also be used to give a direct proof of Theorem~\ref{thm:GKPIntro}, which gives a global analog of \cite{ChenyangPi1} and forms  one of the main results of  \cite{GKP}. We recall (a slight variant of) the statement again: 

\begin{theorem}[{\cite[Theorem 1.5]{GKP}}]
\label{thm:GKP}
Let $X/\mathbf{C}$ be a connected and normal scheme of finite type. Assume there exists some effective $\mathbf{Q}$-Weil divisor $\Delta$ on $X$ such that $(X,\Delta)$ is klt. Let $Z \subset X$ be a closed subset of codimension $\geq 2$. Then there exists a finite cover $f:\widetilde{X} \to X$ of normal connected schemes that is \'etale over $U := X-Z$, and induces an isomorphism $\pi_1(f^{-1}(U)) \xrightarrow{\simeq} \pi_1(\widetilde{X})$.
\end{theorem}

Theorem~\ref{thm:GKP} implies Theorem~\ref{thm:GKPIntro} by setting $Z$ to be the singular locus of $X$. We will deduce Theorem~\ref{thm:GKP} directly from the local result in \cite{ChenyangPi1} (i.e., from Theorem~\ref{thm:XuIntro}) using Theorem~\ref{thm:ConstructibilityMilnorpi1} and formal aspects of the Galois correspondence. In fact, the argument below is a variant of that used in \cite{BCGST} to establish a characteristic $p$ analog of Theorem~\ref{thm:GKP}. We include it here to highlight the utility of non-closed points in such arguments. Indeed, the proof below invokes \cite{ChenyangPi1} at the generic points of the bad locus $Z$; this is necessary because \cite{ChenyangPi1} only applies to complements of points, and is possible in our approach as we work with schemes. In contrast, \cite{GKP} rely on Whitney stratification results from Goresky-Macpherson \cite{GoreskyMacpherson} instead of Theorem~\ref{thm:ConstructibilityMilnorpi1}, and thus do not have access to generic points. Consequently, they also employ certain Lefschetz arguments to reduce to isolated singularities by repeatedly slicing with very general hyperplanes, which ultimately leads to an additional quasi-projectivity hypothesis in their version of Theorem~\ref{thm:GKP}.

\begin{proof}
We prove this by induction on $\dim(Z)$. The base case $\dim(Z) = -1$, corresponding to $Z = \emptyset$, is vacuously true. In general, adopting the notation of Theorem~\ref{thm:ConstructibilityMilnorpi1} and applying its conclusion, we find a dense open $Z_{open} \subset Z$ such that $\pi_1(T_z(Z))$ is constant (as made precise in Theorem~\ref{thm:ConstructibilityMilnorpi1} via specialization maps) as $z$ varies through geometric points of $Z_{open}$. Using this constancy and Theorem~\ref{thm:XuIntro} applied to the generic points of $Z_{open}$, we learn the following: for any geometric point $z \to Z_{open}$, the group $\pi_1(T_z(Z))$ is finite and the image of $\pi_1(T_z(Z)) \to \pi_1(X-Z)$ is independent of $z$ up to conjugation; here we use Proposition~\ref{prop:Localpi1invariance} and the Lefschetz principle to ensure that Theorem~\ref{thm:XuIntro} applies at the generic points mentioned above. Write $Z_{closed} \subset Z$ for the complement of $Z_{open}$, so $\dim(Z_{closed}) < \dim(Z)$. Our strategy is to first solve the extension problem for $X - Z \subset X - Z_{closed}$, and then apply induction to the resulting cover of $X - Z_{closed}$. We will need the following elementary fact:

\begin{lemma}
Given a profinite group $G$ and a finite subgroup $K \subset G$, there exists an open normal subgroup $H \subset G$ such that $H \cap gKg^{-1} = 0$ for all $g \in G$.
\end{lemma}
\begin{proof}
Write $G = \lim Q_i$ as a continuous inverse limit of finite groups. Then $H_i := \ker(G \to Q_i)$ is an open normal subgroup of $G$, and $\cap_i H_i = 0$. This gives $\cap_i (H_i \cap K) = 0$. But then $H_i \cap K = 0$ for $i \gg 0$ as $K$ is finite. As the $H_i$'s are normal, this also yields $H_i \cap gKg^{-1} = 0$ for all $i \gg 0$, as wanted.
\end{proof}

Using this lemma, we may choose choose an open normal subgroup $H \subset \pi_1(U)$ which intersects the finite image of $\pi_1(T_z(Z)) \to \pi_1(U)$ trivially for every geometric point $z \to Z_{open}$. For any finite \'etale map $W \to U$, write $\overline{W} \to X$ for the normalization of $X$ in $W$. Now choose $V \to U$ to be the finite \'etale cover corresponding to $H$, so $H := \pi_1(V) \to \pi_1(U)$ meets the image of $\pi_1(T_z(Z)) \to \pi_1(X)$ trivially for any geometric point $z \to Z_{open}$. By the Galois correspondence, for each such $z$, the base change $V \times_U T_z(Z)$ splits into a disjoint union of copies of the cover $T_z(Z)' \to T_z(Z)$ corresponding to the kernel of the map $\pi_1(T_z(Z)) \to \pi_1(U)$. By the same reasoning applied to further covers of $V$, we learn: if $W \to V$ is a finite \'etale cover, then $W \times_U T_z(Z) \to V \times_U T_z(Z)$ is a disjoint union of sections for any geometric point $z \to Z_{open}$. As normalization commutes with \'etale localization, it follows that for such $W$, the normalized map $\overline{W} \to \overline{V}$ is finite \'etale over all points of $Z_{open}$. Thus, if $g:\overline{V} \to X$ is the structure map, then $\pi_1(V) \xrightarrow{\simeq} \pi_1(\overline{V} - g^{-1}(Z_{closed}))$, and the same is true if $V$ is replaced by a finite \'etale cover of $V$. 

As $\overline{V} \to X$ is \'etale in codimension $1$, the pair $(\overline{V}, g^* \Delta)$ is klt by \cite[Proposition 5.20]{KollarMori}. We also have $\dim(g^{-1}(Z_{closed})) = \dim(Z_{closed}) < \dim(Z)$.  By induction applied the pair $(\overline{V}, g^* \Delta)$ equipped with the closed subset $g^{-1}(Z_{closed})$, we can find a finite cover $h:\widetilde{X} \to \overline{V}$ of normal connected schemes that is \'etale outside $g^{-1}(Z_{closed})$ such that $\pi_1(\widetilde{X} - f^{-1}(Z_{closed})) \xrightarrow{\simeq} \pi_1(\widetilde{X})$, where $f = g \circ h$ is the composite map. As $\widetilde{X} \to \overline{V}$ is \'etale over $V$, we also know from the previous paragraph that $\pi_1(\widetilde{X} - f^{-1}(Z)) \xrightarrow{\simeq} \pi_1(\widetilde{X} - f^{-1}(Z_{closed}))$. Combining these gives the desired statement. 
\end{proof}


\begin{thebibliography}{HKM}
%%%%%%%%%%%%%%%%%%%%%%%%%%%%%%%%%%%%%%%%%%%%%%%%%%%%%%%%%%%%%%%%%%%%%%%%

\bibitem[BCGST]{BCGST}
B.~Bhatt, J.~Carvajal-Rojas, P.~Graf, K.~Schwede, K.~Tucker, \emph{\'Etale fundamental groups of strongly $F$-regular schemes}, Available at \url{https://arxiv.org/abs/1611.03884}

\bibitem[BCHM]{BCHM}
C.~Birkar, P. ~Cascini, C.~Hacon, J.~McKernan, \emph{Existence of minimal models for varieties of log general type}, J. Amer. Math. Soc. 23 (2010), no. 2, 405-468.

\bibitem[Ca]{Cadman}
 C. Cadman, \emph{Using stacks to impose tangency conditions on curves}, Amer. J.
Math. \textbf{129} (2007),  405--427.

\bibitem[CST]{CRST:pi1}
J.~Carvajal-Rojas, K.~Schwede, K.~Tucker, \emph{Fundamental groups of $F$-regular singularities via $F$-signature}, Available at \url{http://arxiv.org/abs/1606.04088}, and to appear in {\em Annales scientifiques de l'ENS.}


\bibitem[Co]{ConradCohomologicalDescent}
B.~Conrad, \emph{Cohomological descent}, Available at \url{http://math.stanford.edu/~conrad/papers/hypercover.pdf}.

\bibitem[De]{DeligneHodgeIII}
P.~Deligne, \emph{Theorie de Hodge III}, Publ. Math. IHES \textbf{44} (1975), pp. 6--77.

\bibitem[EGA]{EGA} J.~Dieudonn\'e and A.~Grothendieck, \emph{\'{E}l\'ements de
    g\'eom\'etrie alg\'ebrique}, Inst. Hautes \'Etudes Sci. Publ. Math. \textbf{4,
    8, 11, 17, 20, 24, 28, 32} (1961--1967).

\bibitem[El]{Elkik}
R. Elkik, \emph{Solutions d'\'equations \`a coefficients dans un anneau hens\'elien}, Ann. Scientifique d'\'E.N.S \textbf{6} (1973), 553--603.

\bibitem[GM]{GoreskyMacpherson}
M.~Goresky, R.~MacPherson, \emph{Stratified Morse theory}, Ergebnisse der Mathematik und ihrer Grenzgebiete (3), volume 14, SpringerVerlag, Berlin, 1988.

\bibitem[GKP]{GKP}
D.~Greb, S.~Kebekus, T.~Peternell, \emph{\'Etale fundamental groups of Kawamata log terminal spaces, flat sheaves, and quotients of abelian varieties}, Duke Math. J., 165 \textbf{10} (2016), 1965-2004.

%\bibitem[HMX]{HMXacc}
%C.~Hacon, J.~McKernan, C.~Xu, \emph{ACC for log canonical thresholds}, Ann. of Math. \textbf{(2)}, 180(2):523--571, 2014

\bibitem[Ha]{HaraRational}
N.~Hara, \emph{A characterization of rational singularities in terms of injectivity of Frobenius maps}, Amer. J. Math. 120, \textbf{5} (1998), 981--996.

\bibitem[HW]{HaraWatanabe}
N.~Hara and K.~Watanabe, \emph{F-regular and F-pure rings vs. log terminal and log canonical singularities}, J. Algebraic Geom. \textbf{11} (2002), no. 2, 363-392.

%\bibitem[HY]{HaraYoshida}
%N.~Hara, K-I.~Yoshida, \emph{A generalization of tight closure and multiplier ideals}, Trans. Amer. Math. Soc. 355, \textbf{8} (2003), 3143--3174.



\bibitem[HH]{HochsterHuneke}
M. Hochster, C. Huneke, \emph{Tight closure and strong F-regularity},  Mem. Soc. Math. France \textbf{38} (1989), 119--133.
 
\bibitem[HL]{HL}
C. Huneke, G. Leuschke, 
\emph{Two theorems about maximal Cohen-Macaulay modules}, 
Math. Ann. \textbf{324} (2002),  391--404. 


\bibitem[LO]{LieblichOlsson}
M. Lieblich, M. Olsson, \emph{Generators and relations for the \'etale fundamental group},  Pure Appl. Math. Q. \textbf{6} (2010),  Special Issue in honor of John Tate,  209--243.


\bibitem[KM]{KollarMori}
J.~Koll\'ar, S.~Mori, \emph{Birational geometry of algebraic varieties}, volume 134 of Cambridge Tracts in Mathematics. Cambridge University Press, Cambridge, 1998

\bibitem[Ko1]{Kollar}
J. Koll\'ar, \emph{New examples of terminal and log canonical singularities}, Available at \url{http://arxiv.org/abs/1107.2864}.

\bibitem[Ko2]{KollarSingMMP}
J.~Koll\'ar, \emph{Singularities of the minimal model program},  Cambridge Tracts in Mathematics, \textbf{200}, Cambridge University Press, Cambridge, 2013.

\bibitem[Ol]{OlssonChow}
M.~Olsson, \emph{On proper coverings of Artin stacks}, Adv. Math, 198: 93-106, 2005.


\bibitem[PST]{PatakfalviSchwedeTucker}
Z.~Patakfalvi, K.~Schwede, K.~Tucker, \emph{Notes for the workshop on positive characteristic algebraic geometry}, Available at \url{https://arxiv.org/abs/1412.2203}



\bibitem[PT]{PolstraTuckerFsig}
T.~Polstra, K.~Tucker, \emph{F-signature and Hilbert-Kunz Multipicity: a combined approach and comparison}, Available at \url{http://arxiv.org/abs/1608.02678}.

\bibitem[Sm]{SmithGuide}
K.~Smith, \emph{Brief Guide to Some of the Literature on $F$-singularities}, Available at \url{http://www.aimath.org/WWN/singularvariety/F-sings.pdf}.

\bibitem[SGA1]{SGA1}
 A.~Grothendieck (ed.), \emph{Rev\'etements \'etales et groupe fondamental}, SpringerVerlag, Berlin, 1971, S\'eminaire de G\'eom\'etrie Alg\'ebrique du Bois Marie 1960-1961 (SGA 1), Dirig\'e par Alexandre Grothendieck. Augment\'e de deux expos\'es de M. Raynaud, Lecture Notes in Mathematics, Vol. 224.

\bibitem[SP]{StacksProject}
\emph{The Stacks Project.} Available at http://stacks.math.columbia.edu.



\bibitem[Ta]{Takagi}
S. Takagi, \emph{An interpretation of multiplier ideals via tight closure}, 
J. Algebraic Geom. \textbf{13} (2004), no. 2, 393--415. 
%
%\bibitem[Tu]{Tucker}
%K. Tucker, unpublished.


\bibitem[Xu]{ChenyangPi1}
C.~Xu, \emph{Finiteness of algebraic fundamental groups}, Compos. Math. 150 (2014), no. \textbf{3}, 409-414.

\end{thebibliography}
\end{document}